\documentclass{mfat}
\pagespan{210}{219}

\usepackage{mmap}
\usepackage[utf8]{inputenc}
\usepackage[all]{xy}     

\newtheorem{theorem}[subsection]{Theorem}
\newtheorem{lemma}[subsection]{Lemma}

\newtheorem*{remark*}{Remark}

\newtheorem{definition}[subsection]{Definition}

\makeatletter \@addtoreset{equation}{section}
\@addtoreset{figure}{section} \@addtoreset{table}{section}
\makeatother

\renewcommand{\theequation}{\thesection.\arabic{equation}}
\renewcommand{\thefigure}{\thesection.\arabic{figure}}
\renewcommand{\thetable}{\thesection.\arabic{table}}

\makeatletter
\newcommand\RRR{\mathbb{R}}

\newcommand\FFF{\mathcal{F}}
\newcommand\id{\mathrm{id}}

\begin{document}

\title[Actions of finite groups and smooth functions on surfaces]
    {Actions of finite groups and smooth functions on surfaces}

\author{Bohdan Feshchenko}
\address{Institute of Mathematics, National Academy of
Sciences of Ukraine, 3 Teresh\-chenkivs'ka, Kyiv, 01601, Ukraine}
\email{fb@imath.kiev.ua}

\date{20/05/2016}

\begin{abstract}
    Let $f:M\to \mathbb{R}$ be a Morse function on a smooth closed surface, $V$ be a connected component of some critical level of $f$, and $\mathcal{E}_V$ be its atom. Let also $\mathcal{S}(f)$ be a stabilizer of the function $f$ under the right action of the group of diffeomorphisms $\mathrm{Diff}(M)$ on the space of smooth functions on $M,$ and $\mathcal{S}_V(f) = \{h\in\mathcal{S}(f)\,| h(V) = V\}.$  The group $\mathcal{S}_V(f)$ acts on the set $\pi_0\partial \mathcal{E}_V$ of connected components of the boundary of $\mathcal{E}_V.$ Therefore we have a  homomorphism $\phi:\mathcal{S}(f)\to \mathrm{Aut}(\pi_0\partial \mathcal{E}_V)$. Let also $G = \phi(\mathcal{S}(f))$ be the image of $\mathcal{S}(f)$ in $\mathrm{Aut}(\pi_0\partial \mathcal{E}_V).$

    Suppose that the inclusion  $\partial \mathcal{E}_V\subset M\setminus V$  induces a bijection $\pi_0 \partial \mathcal{E}_V\to\pi_0(M\setminus V).$
    Let $H$ be a subgroup of $G.$  We present a sufficient condition for existence of a section $s:H\to \mathcal{S}_V(f)$ of the homomorphism $\phi,$ so,  the action of $H$ on $\partial \mathcal{E}_V$ lifts to the $H$-action on $M$ by $f$-preserving diffeomorphisms of $M$.

    This result holds for a larger class of smooth functions $f:M\to \mathbb{R}$
    having the following property: for
    each critical point $z$ of $f$ the germ of $f$ at $z$ is smoothly equivalent to a homogeneous polynomial $\mathbb{R}^2\to \mathbb{R}$
    without multiple linear factors.
\end{abstract}

\maketitle

\section{Introduction}
Let $M$ be a smooth compact  surface. The group of diffeomorphisms  $\mathcal{D}(M)$ of $M$ acts on the space of smooth functions $C^{\infty}(M)$ by the  rule
\begin{equation}\label{equ:main-act}
C^{\infty}(M)\times \mathcal{D}(M)\to C^{\infty}(M),\quad (f,h) = f\circ h.
\end{equation}
The set
$\mathcal{S}(f) = \{h\in\mathcal{D}(M)\;|\; f\circ h = f\}$ is called the {\it stabilizer} of  the function $f$ under   action (\ref{equ:main-act}). Endow  $C^{\infty}(M)$ and  $\mathcal{D}(M)$ with the corresponding   Whitney topologies. The topology on $\mathcal{D}(M)$ induces a certain topology on the stabilizer $\mathcal{S}(f).$

Let $\FFF(M)\subset C^{\infty}(M,\RRR)$ be the set of smooth functions   satisfying the following two conditions:
\begin{itemize}
    \item[(B)]
    the function $f$ takes a constant value at each connected component of $\partial M$, and all critical points of $f$ belong to the interior of $M$;

    \item[(P)]
    for each critical point $x$ of $f$ the germ $(f,x)$ of $f$ at $x$ is smoothly equivalent to some homogeneous polynomial $f_x:\RRR^2\to \RRR$ without multiple linear factors.
\end{itemize}

It is well-known that each homogeneous polynomial $f_x:\RRR^2\to \RRR$ splits into a product of linear and irreducible over $\RRR$ quadratic factors.
Condition (P) means that
\begin{equation}\label{eq:homo}
f_x = \prod_{i = 1}^{n} L_{i}\cdot \prod_{j = 1}^{m} Q_{j}^{q_j},
\end{equation}
where
$L_i(x,y) = a_ix+b_iy$ is a linear form, $Q_j = c_jx^2+d_jxy+e_jy^2$ is an irreducible quadratic form such that $L_i/L_{i'}\neq \mathrm{const}$ for $i\neq i'$, and $Q_j/Q_{j'}\neq \mathrm{const}$ for $j\neq j'$.
So, if $\deg f_x\geq 2$, then $0$ is an isolated critical point of $f_x.$

Recall that if $f:(\mathbb{C},0)\to (\RRR,0)$  is a germ of  $C^{\infty}$-function such that $0\in\RRR^2$ is an isolated critical point of $f$,
then there is a germ  of a homeomorphism $h:(\mathbb{C}, 0)\to (\mathbb{C}, 0)$ such that
$$
{f}_x\circ h(z) = \begin{cases}
\pm|z|^2, & \text{if $0$ is a local extremum, \cite{Danser}},\\
\mathrm{Re}(z^n), & \text{for some $n\in\mathbb{N}$ otherwise, \cite{Prishlyak}}.
\end{cases}
$$
If $0$ is not a local extreme, then the number $n$ does not depend  of a particular  choice of $h.$ In this case the point $0$ will be called a {\it generalized $n$-saddle}, or simply an {\it $n$-saddle}.
The number $n$ corresponds to the number of linear factors  in  (\ref{eq:homo}).
Examples of level sets foliations near isolated  critical points are given in Fig. \ref{fig:ctir-points}.
\begin{figure}[h]
    \center{\includegraphics[width=0.6\linewidth]{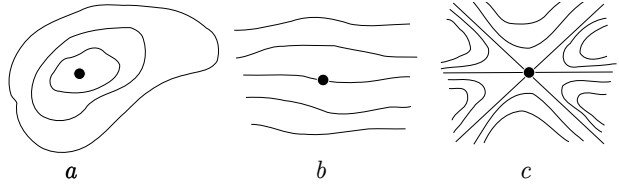}}

    \bigskip

    \caption{Level set foliations in neighborhoods of isolated points
      ((a)  local extreme, (b)  1-saddle, (c)  $3$-saddle)}
    \label{fig:ctir-points}
\end{figure}

Let
$\mathrm{Morse}_{\partial}(M)$ be the space of Morse functions on $M$, which satisfy condition (B),  $f\in\mathrm{Morse}_{\partial}(M),$ and  $x$ be a critical point of $f.$
Then, by Morse Lemma, there exists a coordinate system $(t,s)$ near $x$ such that the function $f$ has one of the following forms $f(s,t) = \pm s^2\pm t^2,$ which are, obviously, homogeneous polynomials without multiple factors.
This implies that $\mathrm{Morse}_{\partial}(M)$ is a subspace of $\FFF(M).$

Let $f\in\FFF(M)$ be a smooth function and $c\in\RRR$ be a real number. A connected component $C$ of the level set $f^{-1}(c)$ is called {\it critical} if it contains at least one critical point, otherwise, $C$ is called {\it regular}. Let $\Delta$ be a foliation of $M$ into connected components of level sets of $f.$ It is well-known that the quotient-space $M/\Delta$ has a structure of $1$-dimensional CW complex. The space $M/\Delta$ is called the Kronrod-Reeb graph, or simply, KR-graph of $f$. We will denote it by $\Gamma_f$. Let $p_f:M\to \Gamma_f$ be a projection of $M$ onto $\Gamma_f$. Then
vertices of $\Gamma_f$ correspond to connected  components of critical level sets of the function $f.$

It should be noted that the function $f\in\FFF(M)$ can be represented as the composition
$$
f = \phi\circ p_f:\;M\xrightarrow{~~p_f~~}\Gamma_f\xrightarrow{~~\phi'~~}\RRR,
$$
where $\phi'$ is the  map induced by $f$.
Let $h\in\mathcal{S}(f)$. Then $f\circ h = f$, and we have $h(f^{-1}(t)) = f^{-1}(t)$ for all $t\in\RRR.$ Hence $h$ interchanges connected components of level sets of the function $f$ and therefore  it induces an automorphism $\rho(h)$ of KR-graph $\Gamma_f$ such that the following diagram is commutative:

\noindent   
In other words, we have a homomorphism
$\rho:\mathcal{S}(f)\to \mathrm{Aut}(\Gamma_f).$ Let  $G = \rho(\mathcal{S}(f))$ be the image of  $\mathcal{S}(f)$ in $\mathrm{Aut}(\Gamma_f).$ It is easy to show that the group $G$ is finite.

Let $v$ be a vertex of $\Gamma_f$ and
$G_v = \{g\in G\;|\;g(v) = v\}$  be the stabilizer of $v$ under the action of  $G$ on $\Gamma_f.$
An arbitrary connected closed $G_v$-invariant neighborhood of $v$ in $\Gamma_f$ containing no other vertices of $\Gamma_f$  will be  called a {\it star} of $v$. We  denote it by $\mathrm{st}(v).$

The set $G_v^{loc} = \{g|_{\mathrm{st}(v)}\;|\;g\in G\}$ which consists of restrictions of elements of $G_v$ onto the star $\mathrm{st}(v)$
is a subgroup of $\mathrm{Aut}(\mathrm{st}(v))$. This group will be called a {\it local stabilizer} of $v.$ Let also $r:G_v\to G_v^{loc}$ be the map defined by $r(g) = g|_{\mathrm{st}(v)}$ for $g\in G_v$, i.e., $r$ is the restriction map.

Let $v$ be a vertex of $\Gamma_f$, and $V = p_f^{-1}(v)$ be the corresponding connected component of the critical level set $f^{-1}(p_f(v)).$

\begin{definition}
    A vertex $v$ of the graph $\Gamma_f$ will be called {\bf special} if there is a bijection between connected components of
    $\mathrm{st}(v)\setminus v$ and $M\setminus V$.
    The corresponding connected component $V = p_f^{-1}(v)$  will be called {\bf special}.
\end{definition}

It follows from  definition of KR-graph $\Gamma_f$ that for a special vertex $v$ there is a 1--1 correspondence between connected components of complement to $v$ in $\mathrm{st}(v)$ and connected components of $\Gamma_f\setminus v.$

Note that a special component $V$ gives a partition $\Xi$ of the surface $M$ whose $0$-dimensional elements are vertices of $V,$ $1$-dimensional elements are edges of $V$, and $2$-dimensional elements are connected components of complement of $V$ in $M.$ Since $M$ is compact, it follows that $\Xi$ has a  finite number of elements in each dimension.
\section{Main result}
Let $f\in\FFF(M)$. Suppose that its Kronrod-Reeb graph $\Gamma_f$ contains a special vertex $v$, and $V$ be the special component of level set of $f$ which corresponds to $v.$

Let $\mathcal{S}_V(f) = \{h\in \mathcal{S}(f)\;|\;h(V) = V\}$ be a subgroup of $\mathcal{S}(f)$ leaving $V$ invariant.  It is easy to see that $\rho(\mathcal{S}_V(f))\subset G_v.$ We denote by $\phi$ the  map
$$
\phi = r\circ \rho: \mathcal{S}_V(f)\xrightarrow{~~\rho~~} G_v\xrightarrow{~~r~~} G_v^{loc}.
$$

Let $H$ be a subgroup of $G_v^{loc}$ and $\mathcal{H}= \phi^{-1}(H)$ be a subgroup of $\mathcal{S}_V(f).$ We will say that the group
$\mathcal{H}$ has property {\rm (C)}  if the following conditions hold.
\begin{itemize}
    \item[(C)] Let $h\in\mathcal{H},$ and $E$  be a $2$-dimensional element of $\Xi$. Suppose that $h(E) = E.$ Then $h(e) = e$ for all other $e\in\Xi$, and the map $h$ preserves orientation of each element of $\Xi.$

\end{itemize}
\begin{lemma}\label{lm:act}
    If $\mathcal{H} = \phi^{-1}(H)$ has property {\rm (C)}, then $H$ acts on the set of all elements of the partition $\Xi.$ Moreover this action is free on the set of $2$-dimensional elements of $\Xi.$
\end{lemma}
\begin{proof}
    Let $g\in H$, and $h \in\mathcal{H}$ be a diffeomorphism such that $\phi(h) = g.$ Define the map
    $
    \tau:H\times\Xi\to \Xi
    $
    by the following rule
    $$
    \tau(g,e) = h(e),\quad e\in\Xi.
    $$

    We claim that this definition does not depend of a particular choice of such $h.$
    Let $h_1,h_2\in\mathcal{H}$ be diffeomorphisms such that $\phi(h_1) = \phi(h_2).$ Then $\phi(h_1\circ h_2^{-1}) = 1_H,$ where $1_H$ be the unit of $H.$ By definition of the unit $1_H$, we have
    $(h_1\circ h^{-1}_2)(E) = E$ for each $2$-dimensional component $E$ of $\Xi.$ Then, by condition (C), $(h_1\circ h_2^{-1})(e) = e$ for other $e\in \Xi.$ Hence $h_1(e) = h_2(e).$ So, the map $\tau$ is well-defined.
    It is easy to see that $\tau(1_H,e) = e,$ where $1_H$ is the unit of $H,$ and $\tau(g_1,\tau(g_2,e)) = \tau(g_1\circ g_2,e)$ for each $g_1,g_2\in H,$ and $e\in\Xi.$
    Thus $\tau$ is an $H$-action on $\Xi.$

    Suppose  $h\in \mathcal{H}$ is such that $h(E) = E$ for some $2$-dimensional component $E$ of $\Xi.$ Then, by condition (C),  $h(E') = E'$ for each $2$-dimensional component $E'$ of $\Xi.$ Hence, $h = \id,$ so the ${H}$-action on the set  of $2$-dimensional components of $\Xi$ is free.
\end{proof}

Thus condition (C) implies that $H$ $<<$combinatorially$>>$ acts of $M$
i.e., it ensures  invariance of the partition $\Xi$ under the action of $H$ on $M.$ Our aim is to prove that in fact this <<combinatorial>> action is induced by a real action of $H$ on $M$ by diffeomorphisms preserving $f$.

Namely the following theorem holds.
\begin{theorem}\label{th:main}
    Suppose $f\in\FFF(M)$ is  such that its KR-graph $\Gamma_f$ contains a special vertex $v$, and $G_v^{loc}$ be the local stabilizer of $v.$ Let also $H$ be a subgroup of $G_v^{loc},$ and $\mathcal{H} = \phi^{-1}(H)$ be a subgroup of $\mathcal{S}_V(f)$ satisfying condition $(\mathrm{C}).$ Then there exists a section $s:H\to\mathcal{H}$ of the map $\phi,$ i.e., the map $s$ is a homomorphism satisfying the  condition $\phi\circ s = \id_{H}.$
\end{theorem}
Group actions which have the property of invariance of some partition of the surfaces are studied by Bolsinov and Fomenko \cite{BolsinovFomenko:1998}, Brailov \cite{Brailov:1998}, Brailov and Kudryavtseva \cite{BrailovKudryavtseva:VMNU:1999}, Kudryavtseva \cite{Kudryavtseva:MatSb:1999},
Maksymenko \cite{Maksymenko:def:2009}, Kudryavtseva and Fomenko \cite{KudryavtsevaFomenko:DANRAN:2012, KudryavtsevaFomenko:VMU:2013}.

\subsection{Structure of the paper} In Section \ref{sec:Au} we recall the definitions and statements that will be used in the text. The topological structure of the atom $\mathcal{E}_{V,\varepsilon}$ which corresponds to $V$ is described in Section \ref{sec:top}.
In section \ref{sec:proof}, we construct an $H$-action on the surface~$M.$

\section{Symmetries of homogeneous polynomials}\label{sec:Au}
Let $f_z:\RRR^2\to \RRR$ be a homogeneous polynomial without multiple linear factors. Suppose the origin $0\in\RRR^2$ is not a local extreme for $f_z.$ Let also
$\mathcal{L}(f_z)$ be a group of orientation preserving linear automorphisms $h:\RRR^2\to \RRR^2$ such that $f_z\circ h = f_z$.
The following lemma holds:
\begin{lemma}{\rm(}\cite{Maksymenko:connected-components:2009}, {\rm{Section 6}}{\rm).}\label{lm:ho}
    After some linear change of coordinates one can assume that
    \begin{enumerate}
        \item   if $\deg f_z = 2$, then the group $\mathcal{L}(f_z)$ consists of the linear transformations of the following form
        $$
        \pm\begin{pmatrix}
        a& 0 \\
        0 &  \frac{1}{a}
        \end{pmatrix},\quad a>0,
        $$
        see  \cite[Section 6, case (B)]{Maksymenko:connected-components:2009};
        \item if $\deg\geq 3,$ then the group of $\mathcal{L}(f_z)$ is a finite cyclic subgroup of $\mathrm{SO}(2),$
        \cite[Section~6, case (E)]{Maksymenko:connected-components:2009}.
    \end{enumerate}
\end{lemma}

We will also need the following lemma:
\begin{lemma}{\rm(}\cite{Maksymenko:connected-components:2009}, {\rm{Corollary 7.4).}}\label{lm:tang}
    Let  $h:(\RRR^2,0)\to (\RRR^2,0)$ be a germ of a diffeomorphism $h:\RRR^2\to \RRR^2$ at $0\in\RRR^2,$ and $T_0h$ be its tangent map at $0\in \RRR^2.$ If $f_z\circ h = f_z$, then $f_z\circ T_0h = f_z.$
\end{lemma}

\begin{proof}
    For the sake of completeness we will recall a short
    proof from \cite{Maksymenko:connected-components:2009}.

    Assume that the polynomial $f_z$ is a homogeneous function of degree $k,$ i.e., $f_z(tx) = t^kf_z(x)$ for $t\geq 0$ and $x\in \RRR^2.$
    Then
    $$
    f_z(x) = \frac{f_z(tx)}{t^k} = \frac{f_z(h(tx))}{t^k} = f_z\left( \frac{h(tx)}{t} \right) \xrightarrow[~~t\to 0~~]{} (f_z\circ T_0h)(x).
    $$
    Lemma \ref{lm:tang} is proved.
\end{proof}

\section{Topological structure of the atom $\mathcal{E}_{V,a}$}\label{sec:top}
Let $f$ be a smooth function from $\FFF(M)$, $\varepsilon_1 >0,$ $c\in \RRR,$ and $V$ be a connected component of some critical level $f^{-1}(c)$  of $f$.

Let also $\mathcal{E}$ be a connected component of $f^{-1}([c-\varepsilon_1,c+\varepsilon_1]),$ which contains $V.$ Assume that the boundary $\partial \mathcal{E}$ consists of $n+k$ connected components $A_i,$ $i =1,2,\ldots, n+k,$ i.e., $\partial\mathcal{E} = \bigcup_{i = 1}^n A_i\cup\bigcup_{j = 1}^kA_{-j}.$ Since $f\in\FFF(M)$, it follows that $f|_{\mathcal{E}}$  belongs to $\FFF(\mathcal{E})$, and so, by (B), $f|_{\mathcal{E}}$ takes a constant value at each connected component of the boundary $\partial \mathcal{E}$. Assume that $f(A_i) = c_i\in[c,c+\varepsilon_1]$, $i\geq 1$, and $f(A_{i}) = d_i\in[c-\varepsilon_1,c],$ $i\leq 1.$ Put $c' = \min\{c_i\}$ and $d' = \max\{d_i\}$. Fix $a>0$  such that
$[c-a,c+a]\subset [d',c'].$

A connected component of $f^{-1}([c-a,c+a]),$ which contains $V$  will be called an {\it atom} of $V$ and denoted by $\mathcal{E}_{V,a}$.

Let $H$ be a subgroup of $G_v^{loc}$ and $\mathcal{H}= \phi^{-1}(H)\subset\mathcal{S}_V(f).$
We will need the following lemma.
\begin{lemma}\label{lm:part}
    Let $\mathcal{E}_{V,a}$ be an atom of a special critical component $V$, $A$ be a connected component  of $\partial \mathcal{E}_{V,a}$, and $h\in\mathcal{H}.$ Assume that the group $\mathcal{H}$ has property  (C). If $h|_{\mathcal{E}_{V,a}}(A) = A,$ then $h$ preserves the orientation of $A$.
\end{lemma}
\begin{proof}
    Fix a Riemannian metric $\langle\cdot,\cdot\rangle$ on $M.$  Let $\nabla f$ be a gradient vector field of the function $f$ in this Riemannian metric. Let also $Q$ be a set of points $x\in A$ such that there exists an integral curve $c_x$ of $\nabla f,$ which joins the point $x$ with some point $y_x\in V.$ Then $Q$ is a union of open intervals in $A,$
    and the map $\psi:Q\to V$, $\psi(x) = y_x$ is an embedding. The image of $\psi(Q)$ is a cycle  in $V.$ So, the connected component $A$ of $\partial \mathcal{E}_{V,a}$ defines the cycle $\gamma_A$ in $V.$ Moreover the orientation of $A$ induces the orientation of $\gamma_A$ and vice versa, see   \cite{Oshemkov}.

    Assume that $\mathcal{H}$ has property (C). Let $h\in \mathcal{H}$ and $E$ be a $2$-dimensional element of $\Xi$ such that $h(E) = E.$ Then by (C), $h(e) = e$ for all other $e\in \Xi.$ In particular $h(\gamma_A) = \gamma_A$ and $h$  preserves orientation  of $\gamma_A$. Then $h(A) = A,$ and $h$ preserves  orientation of $A$.

\end{proof}

\section{Proof of Theorem \ref{th:main}}\label{sec:proof}
Suppose $f\in\FFF(M)$ is such that  its KR-graph contains a special vertex
$v,$ $V = p^{-1}_f(v)$ be the corresponding special  component of some level set, which corresponds to $v$, and $G_v^{loc}$ be the local stabilizer of $v.$

Let $H$ be a subgroup of $G_v^{loc}$ such that $\mathcal{H} = \phi^{-1}(H)$ has property (C).
We will construct a lifting of the $H$-action on $\mathrm{st}(v)$ to the action $\Sigma:{H}\times M\to M$ of the group $H$ on the surface $M.$

By Lemma \ref{lm:act} there is an action $\sigma^0:H\times \mathsf{V}\to \mathsf{V}$ of $H$ on the set of vertices of $V$ defined by the rule:
$$
\sigma^0(g,z) = h(z),
$$
where $h\in\mathcal{H}$ is any diffeomorphism such that $\phi(h) = g.$

\subsection*{Step 1}\label{sub:step1} Now we will extend the action $\sigma^0$ to the $H$-action $\sigma^1$ on the set  of neighborhoods of vertices of $V.$ Assume that the action $\sigma^0$ has $s$ orbits $\mathsf{V}_r = \{z_{r0},z_{r1},\ldots, z_{r\,k(r)}\}$ for some $k(r)\in\mathbb{N},$ $r = 1,2,\ldots, s,$ and let $\mathsf{V} = \bigcup_{r = 1}^s\mathsf{V}_r$ be the union of vertices of $V.$

Then, by definition of the class $\FFF(M),$ for each $r = 1,2,\ldots, s$ there exists a chart $(U_{r0},q_{r0})$ which contains $z_{r0}$ such that the map $f\circ q_{r0}^{-1} = f_{r}$ is a homogeneous polynomial without multiple linear factors. We can also assume that $q_{r0}(U_{r0})\subset \RRR^2$ is a $2$-disk with the center at $0\in\RRR^2$ and radius $\varepsilon$, and   the group $\mathcal{L}(f_r)$ has the properties described in Lemma \ref{lm:ho}.  Fix any  diffeomorphisms
\begin{equation}\label{eq:h}
h_{ri}\in \mathcal{H}\; \quad\text{such that}\quad  h_{ri}(z_{r0}) = z_{ri},\quad  i = 1,2,\ldots, k(r),
\end{equation}
and define charts $(U_{ri},q_{ri})$ for the points $z_{ri}$, $i = 1,2,\ldots,k(r)$ in the following
way:
\begin{itemize}
    \item $U_{ri} = h_{ri}(U_{r0})$;
    \item the map $q_{ri}$ is defined from the diagram:
    $$
    \xymatrix{
        U_{r0}\ar[d]_{q_{r0}} \ar[r]^{h_{ri}} & U_{ri} \ar[dl]^{q_{ri}}\\
        q_{r0}(U_{r0})
    }
    $$
    i.e., $q_{ri} = q_{r0}\circ h_{ri}^{-1}$.
\end{itemize}
Reducing $\varepsilon$, we can assume that $U_{ri}\cap U_{rj} = \varnothing$ for $i\neq j.$

Thus  the chart $(U_{ri}, q_{ri})$ is chosen so that the map $ f\circ q_{ri}^{-1}:q_{r0}(U_{r0})\to \RRR$ is a homogeneous polynomial without multiple linear factors which coincides with  given polynomial $f_r$ for the chart $(U_{r0},q_{r0}).$ We also put $\mathsf{U}_r = \bigcup_{i = 0}^{k(r)}U_{ri}$, and $\mathsf{U} = \bigcup_{r = 1}^s\mathsf{U}_{r}$.
\begin{lemma}
    There exist a homomorphism $\lambda_1:\mathcal{H}\to\mathrm{Diff}(\mathsf{U})$ and a monomorphism $\chi_1:H\to \mathrm{Diff}(\mathsf{U})$ such that
    the following diagram is commutative:
    $$
    \xymatrix{
        \mathcal{H}\ar[rr]^-{\lambda_1} \ar[d]_{\phi} && \mathrm{Diff}(\mathsf{U})\\
        H \ar[urr]_-{\chi_1}&
    }
    $$

\end{lemma}
\begin{proof}
    (1) First we construct a map $\lambda_1.$ Let $h\in\mathcal{H}$ be such that $h(z_{ri}) = z_{rj}$ for some $i,j = 0,1,\ldots k(r)$ and $r = 1,2,\ldots, s.$ Let also $\gamma_h = q_{rj}\circ h\circ q_{ri}^{-1}$ be a diffeomorphism of $q_{r0}(U_{r0}).$
    It is easy to see that the map $\gamma_{h}$ preserves the polynomial $f_r.$ By Lemma \ref{lm:tang} the tangent map $T_0\gamma_{h}$ also preserves the polynomial $f_r$, so $T_0\gamma_{h}\in\mathcal{L}(f_r).$ Define a linear map $A\in\mathcal{L}(f_r)$ as follows:
    if $\deg f_r = 2$, then, by Lemma \ref{lm:ho},
    $$
    T_0\gamma_{h} =
    \begin{pmatrix}
    a & 0\\
    0 & \frac{1}{a}
    \end{pmatrix},\quad a\neq0,
    $$
    and we set
    $$
    A_{h} = \mathrm{sign}(a)
    \begin{pmatrix}
    1, & 0\\
    0 & 1
    \end{pmatrix}.
    $$
    If $\deg f_z \geq 3,$ then by assumption and Lemma \ref{lm:ho}, $\mathcal{L}(f_z)$ is a cyclic subgroup of $\mathrm{SO}(2).$ In this case we put
    $$
    A_{h} = T_0\gamma_{h}.
    $$
    We define the diffeomorphism $\lambda_1(h)\in\mathrm{Diff}(\mathsf{U})$ by the rule:
    \begin{equation}\label{eq:def-lambda1}
    \lambda_1(h)|_{U_{ri}} = q_{rj}^{-1}\circ A_h\circ q_{ri}.
    \end{equation}
    (2) Now we prove that the map $\lambda_1$ is a homomorphism. Suppose $h_1, h_2\in\mathcal{H}$ are such that $h(z_{ri}) = z_{rj}$ and $h(z_{rj}) = z_{rk}.$  By (\ref{eq:def-lambda1}), we have
    $$
    \lambda_1(h_1)|_{U_{ri}} = q_{rj}^{-1}\circ A_{h_1}\circ q_{ri},\quad \lambda_1(h_2)|_{U_{rj}} = q_{rk}^{-1}\circ A_{h_2}\circ q_{rj},
    $$
    and
    $$
    \lambda_1(h_2)|_{U_{rj}}\circ \lambda_1(h_1)|_{U_{ri}} = q_{rk}^{-1}\circ A_{h_2}\circ A_{h_1}\circ q_{ri}.
    $$
    On  the other hand, we have
    $$
    \lambda_1(h_2\circ h_1) = q_{rk}^{-1}\circ A_{h_2\circ h_1}\circ q_{ri}.
    $$
    It follows from the  definition of the linear map $A_h$, that $A_{h_2\circ h_1} = A_{h_2}\circ A_{h_1}.$ Hence
    $$
    \lambda_1(h_2\circ h_1) = \lambda_1(h_2)\circ \lambda_1(h_1).
    $$
    So, the map $\lambda_1$ is a homomorphism.

    (3) Let $g\in H$ and $h\in\mathcal{H}$ be such that $\phi(h) = g.$ Then we define the map
    $\chi_1:H\to \mathrm{Diff}(\mathsf{U})$ by the rule
    $$
    \chi_1(g) = \lambda_1(h).
    $$
    Obviously that $\chi_1$ is a homomorphism. It remains to prove that the map $\chi_1$ is a monomorphism. It is sufficient to check that $\mathrm{Ker}\chi_1 = \mathrm{Ker}\psi,$ i.e., $\lambda_1(h) = \id_{\mathsf{U}}$ iff $h$ trivially acts on the set of $2$-dimensional elements of $\Xi.$

    Suppose that $h$ trivially acts on the set of $2$-dimensional elements of $\Xi.$ By condition (C), $h$ trivially acts on set of vertices and edges of $V$.  Since $h(z_{ri}) = z_{ri}$ for all $i = 0,1,\ldots k(r)$ and $r = 1,2,\ldots, s$, it follows from (\ref{eq:def-lambda1}) that $\lambda_1(h) = \id_{\mathsf{U}}.$

    Suppose  $h\in\mathcal{H}$ is such that $\lambda_1(h) = \id_{\mathsf{U}}.$ Then $h(e) = e$ for each edge $e$ of $V$,
    and $h$ preserves the orientation of $e.$ Hence by Lemma \ref{lm:part}, $h$ leaves invariant each connected
    component of $\partial \mathcal{E}_{V,a}$ with its orientation. Therefore $h$ trivially acts on
    the set of   \break   $2$-dimensional elements of $\Xi.$
\end{proof}

Let $\sigma^1:H\times \mathsf{U}\to \mathsf{U}$ be a map defined by the formula
$$
\sigma^1(g,x) = \chi_1(g)(x),
\quad x\in \mathsf{U}.
$$
Since $\chi_1$ is a homomorphism, it follows  that $\sigma^1$ is an $H$-action on $\mathsf{U}.$

\subsection*{Step 2} In this step we extend the action $\sigma^1$ to the  $H$-action $\sigma$ on the atom $\mathcal{E}_{V,a}.$ We start with some preliminaries. Let  $(U_{ri},q_{ri})$ be the chart on $M,$ which contains $z_{ri},$ defined above. The projection map $q_{ri}$ induces the map $Tq_{ri}:TU_{ri}\to Tq_{ri}(U_{ri})$ between tangent bundles of $U_{ri}$ and $q_{ri}(U_{ri})\subset \RRR^2.$ Fix a Riemannian metric $\langle\cdot,\cdot\rangle$ on $M$ such that the following diagram is commutative
$$
\xymatrix{
    TU_{ri} \ar[rr]^{Tq_{ri}}  && Tq_{ri}(U_{ri}) \\
    U_{ri} \ar[u]^{\nabla f|_{U_{ri}}} \ar[rr]_{q_{ri}} && q_{ri}(U_{ri}) \ar[u]_{\nabla f_r|_{q_{ri}(U_{ri})}}
}
$$
where $\nabla f$ and $\nabla f_r$ are gradient fields of $f$ and $f_r$ in Riemannian metrics on $M$ and on $\RRR^2$ respectively. Let also $\mathbf{G}$ be the flow of $\nabla f$ on $M.$

\subsection*{Another description of the diffeomorphism $\lambda_1(h)$}
Let $x\in U_{ri}$ be a point, $i = 0,1,2,\ldots, k(r)$, $r = 1,2,\ldots s,$ and $y = \lambda_1(h)(x)$ be its image under $\lambda_1(h).$ Let also $\omega_x$ and $\omega_y$ be the trajectories of the gradient flow $\mathbf{G}$ such that $x\in\omega_x$ and $y\in\omega_y.$ Since $\lambda_1(h)$ preserves trajectories of the flow $\mathbf{G}$ in $\mathsf{U}$, it follows that $\lambda_1(h)(\omega_x\cap U_{ri}) = \omega_y\cap \lambda_1(h)(U_{ri}).$ By definition of $\lambda_1(h)$ we have that $f(x) = f(y).$ In particular, if the trajectory $\omega_x$ intersects some edge $R$ of $V$ at some point  $x',$ and $y' = \lambda_1(h)(x'),$ then $y = f^{-1}(f(x))\cap \omega_{y'},$ where $\omega_{y'}$ is the trajectory of $\mathbf{G}$, which passes through the point $y'$. Namely the image of $x$ is uniquely defined by the image of the point $x'.$

By Lemma \ref{lm:act}, the group $H$ acts on the set of all edges $\mathsf{R}$ of $V$. Assume that this action has $u$ orbits $\mathsf{R}_r = \{R_{r0},R_{r1},\ldots R_{r\, n(u)}\}$ for some $n(u)\in\mathbb{N}$ and $r = 1,2,\ldots, u.$ We also put $\mathsf{R} = \bigcup_{r = 1}^u\mathsf{R}_r.$
For each edge $R_{ri}$ fix
\begin{itemize}
    \item[(I)] a $C^{\infty}$-diffeomorphism $\ell_{ri}:(-1,1)\to R_{ri}$
    such that  restrictions $\ell_{ri}|_{(-1,-1+\varepsilon)}$ and $\ell_{ri}|_{(1-\varepsilon,1)}$ are isometries,
\end{itemize}
where $\varepsilon$ is the  radius of the disk $q_{r0}(U_{r0})$ defined in Step 1.

\begin{lemma}
    There exist a homomorphism $\lambda_2:\mathcal{H}\to \mathrm{Diff}(\mathcal{E}_{V,a})$ and a monomorphism $\chi_2:H\to \mathrm{Diff}(\mathcal{E}_{V,a})$ such that the following diagram is commutative:
    $$
    \xymatrix{
        \mathcal{H}\ar[rr]^-{\lambda_2} \ar[d]_{\phi} && \mathrm{Diff}(\mathcal{E}_{V,a})\\
        H \ar[urr]_-{\chi_2}&
    }
    $$
    and $\lambda_2(h)|_{\mathsf{U}} = \lambda_1(h)|_{\mathsf{U}\cap \mathcal{E}_{V,a}}.$
\end{lemma}
\begin{proof}
    Let $h\in\mathcal{H}.$  We will extend the diffeomorphism $\lambda_1(h)$ to a diffeomorphism $\lambda_2(h)$ of the atom $\mathcal{E}_{V,a}$.
    Let $x\in \mathcal{E}_{V,a}$ be any point. If $x\in U_{ri}$ for some $i = 0,1,\ldots k(r)$, $r = 1,2,\ldots s,$ then we put
    $\lambda_2(h)(x) :=\lambda_1(h)(x).$

    Suppose that $x\not \in \mathsf{U}$. Let $\omega_x$ be a trajectory of the flow $\mathbf{G}$ passing through the point $x$. Then we have one of the following two cases: the trajectory $\omega_x$ either
    \begin{enumerate}
        \item intersects  some edge $R$ of $V$ at a point, say $y$, or
        \item converges to some vertex $z$ of $V$.
    \end{enumerate}

    In the case (1)
    let $R' = h(R)$, $\ell:(-1,1)\to R$ and $\ell':(-1,1)\to R'$ be maps, defined by  (I) for $R$ and $R'$ respectively, and
    $$
    h' = \ell{'}\circ \ell^{-1}:  R \xrightarrow{~~\ell^{-1}~~}(-1,1)\xrightarrow{~~\ell'~~} R'.
    $$
    Let also $y' = h'(y)\in R'$,  $\omega_{y'}$ be the trajectory of $\mathbf{G},$ which passes through $y',$ and $x'$ be a unique point in $\omega_{y'}$ such that $f(x) = f(x')$.
    Then we put $\lambda_2(h)(x) = x'.$

    Consider the case (2).
    Let $U$ be the neighborhood of $z$, defined in Step 1,  $z' = \lambda_1(h)(z)$ be the corresponding point in $U'= \lambda_1(h)(U),$  $\omega_{z'}$ be the trajectory of $\mathbf{G}$ such that $\omega_{z'}\cap U' = \lambda_1(h)(\omega_x \cap U),$ and $x'$ be a unique point in $\omega_{z'}$ such that $f(x) = f(x')$. In this case we define $\lambda_2(h)$ by the rule: $\lambda_2(h)(x) = x'$.

    By definition $\lambda_2(h)|_{\mathsf{U}} = \lambda_1(h)|_{\mathsf{U}\cap \mathcal{E}_{V,a}}.$ Let $\chi_2:H\to \mathrm{Diff}(\mathcal{E}_{V,a})$ be the map defined as follows: for $g\in H$ and $h\in \mathcal{H}$ such that $\phi(h) = g$, we put $\chi_2(g) = \lambda_2(h).$ It is easy to check that the map $\lambda_2$ is a homomorphism. Moreover $\lambda_2(h) = \id_{\mathcal{E}_{V,a}}$ iff $\lambda_1(h) = \id_{\mathsf{U}}.$ Therefore $\chi_2$ is a monomorphism.
\end{proof}
Define the map $\sigma:H\times \mathcal{E}_{V,a}\to \mathcal{E}_{V,a}$ by the rule
  $$
\sigma(g,x) = \chi_2(g)(x).
$$
Since $\chi_2$ is the homomorphism, it follows that the map $\sigma$ is an $H$-action on the atom $\mathcal{E}_{V,a}.$

\subsection*{Step 3}
In this step we extend the $H$-action $\sigma$ on the atom $\mathcal{E}_{V,a}$ to the $H$-action on the surface $M.$
We start with some preliminaries. Let $\mathsf{E}$ be a set of $2$-dimensional elements of $\Xi.$ By Lemma \ref{lm:act} the group $H$ acts on the set $\mathsf{E}$. Assume that this action has $y$ orbits $\mathsf{E}_r = \{E_{r0},E_{r1},\ldots, E_{r\,k(r)}\}$, $i = 0,1,\ldots,k(r)$, and $r = 1,2,\ldots y.$ We also put $\mathsf{E} = \bigcup_{r = 1}^y \mathsf{E}_r.$
Fix diffeomorphisms $h_{ri}\in\mathcal{H}$ such that $h_{ri}(E_{r0}) = E_{ri}.$

Let $Y_r = E_{r0}\cap f^{-1}([-a,-a/2]\cup [a/2,a])\cap \mathcal{E}_{V,a}.$ Since $v$ is a special vertex, it follows that the set $Y_r$ is a cylinder. We put $Y_{ri} = h_{ri}(Y_r)$, and $Y = \bigcup_{r = 1}^y\bigcup_{i = 0}^{k(r)}Y_{ri}.$

We choose $a_1 > a$ such that the set $\mathcal{E}_{V,a_1}$ is also an atom of $V$. Let
$$
Z_r = E_{r0}\cap f^{-1}([-a_1,a/2]\cup [a/2, a_1])\cap \mathcal{E}_{V,a_1}.
$$
By definition, we have that $Y_r\subset Z_r,$ and $Z_r$ does not contain critical points of $f$. We also put $Z_{ri} = h_{ri}(Y_r)$ and $Z  = \bigcup_{r = 1}^y\bigcup_{i = 0}^{k(r)}Z_{ri}.$

\begin{figure}[h]
    \center{\includegraphics[width=0.4\linewidth]{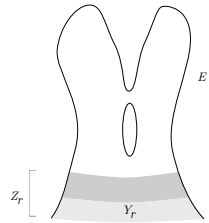}}
    \caption{The $2$-dimensional component $E_{r0}$, and its subsets $Y_r$ and $Z_r$.}
    \label{fig:hC_not_C}
\end{figure}

Fix a vector field $F$ on $Z$ such that its orbits coincide with connected components of level sets of the restriction $f|_{Z},$ and let $\mathbf{F}$ be the flow of $F.$
Then for each smooth function $\alpha\in C^{\infty}(M)$ we can define the following map
$$
\mathbf{F}_{\alpha}:M\to M,\quad \mathbf{F}_{\alpha}(x) = \mathbf{F}(x,\alpha(x)).
$$
Such maps have been studied in \cite{Maktymenko:smooth-shifts:2003}.

Since all orbits of $\mathbf{F}$ are closed, it follows from \cite[Theorem 19]{Maktymenko:smooth-shifts:2003} that the map $\mathbf{F}_{\alpha}$ is a diffeomorphism, iff the Lie derivative $F\alpha$ of $\alpha$ along $F$ satisfies the condition:
$F\alpha > -1.$ Moreover we have that $(\mathbf{F}_{\alpha})^{-1} = \mathbf{F}_{\xi}$, where
\begin{equation}\label{eq;invers}
\xi = -\alpha\circ \mathbf{F}_{\alpha}^{-1}.
\end{equation}

\begin{lemma}
    For each $g\in H$ the map $\chi_2(g)$ extends  to a diffeomorphism $\Sigma(g)\in\mathcal{S}(f)$, so that  the correspondence $g\mapsto \Sigma(g)$ is a homomorphism   $\Sigma:H\to \mathcal{S}(f)$.
\end{lemma}
\begin{proof}
    We will need the following two lemmas.
    \begin{lemma}\label{lm:ch2}
        Let  $g\in H$ and $h\in \mathcal{H}$ be such that $\phi(h) = g$, and  $h(E_{r0}) = E_{ri}.$ Then there exists a unique $C^{\infty}$-function $\xi_{ri}:Y_{r}\to \RRR$ such that
        $$
        \chi_2(g)|_{Y_r} = h_{ri}|_{Y_r}\circ \mathbf{F}_{\xi_{ri}}:Y_{r}\to Y_r.
        $$
        In particular, the function $\xi_{ri}$ depends only on $g.$
    \end{lemma}
    \begin{lemma}\label{lm:ext}
        The diffeomorphism $\mathbf{F}_{\xi_{ri}}$ extends to a diffeomorphism $w_{ri}:E_{r0}\to E_{r0}$ such that $f\circ w_{ri} = f$ on $E_{r0}.$
    \end{lemma}
    \noindent We prove Lemma \ref{lm:ch2} and Lemma \ref{lm:ext} bellow, and now we will complete Theorem \ref{th:main}.

    Define a diffeomorphism $\tilde{h}_{ri}:E_{r0}\to E_{ri}$  by the formula:
    $$
    \tilde{h}_{ri} = h_{ri}\circ w_{ri}.
    $$

    Let $h\in\mathcal{H}$. Define the diffeomorphism $\lambda_3(h)$ by the rule: if $h(E_{ri}) = E_{rj},$ then
    $$
    \lambda_3(h)|_{E_{ri}} = \tilde{h}_{rj}\circ \tilde{h}_{ri}^{-1}:E_{ri}\xrightarrow{~~\tilde{h}_{ri}^{-1}~~} E_{r0}\xrightarrow{~~\tilde{h}_{rj}~~} E_{rj}.
    $$
    It follows from Lemma \ref{lm:ext} that $\lambda_3(h)$ coincides with $\lambda_2(h)$ on $Y.$

    Now we will check that the correspondence $h\mapsto\lambda_3(h)$ is a homomorphism. Let $h_1$ and $h_2$ be homeomorphisms from $\mathcal{H}$ such that $h_1(E_{ri}) = E_{rj}$ and $h_2(E_{rj}) = E_{rk}.$ By definition $\lambda_3(h_1)|_{E_{ri}} = \tilde{h}_{rj}\circ \tilde{h}^{-1}_{ri}$, and $\lambda_3(h_2)|_{E_{rj}} = \tilde{h}_{rk}\circ \tilde{h}_{rj}^{-1}.$ Then $\lambda_3(h_2)|_{E_{rj}}\circ \lambda_3(h_1)|_{E_{ri}} = \tilde{h}_{rk}\circ \tilde{h}_{rj}^{-1}\circ \tilde{h}_{rj}\circ \tilde{h}^{-1}_{ri} = \tilde{h}_{rk}\circ\tilde{h}_{ri}^{-1} = \lambda_3(h_2\circ h_1)|_{E_{ri}}.$
    Hence, the map $\lambda_3$ is a homomorphism.

    Let $g\in H$, and $h\in\mathcal{H}$ be such that $\phi(h) = g.$ By condition (C), $\lambda_3(h) = \id$ iff $h(E_{ri}) = E_{ri}$ for some $r = 1,2,\ldots y$, and $i = 0,1,\ldots k(r).$ So the map $\chi_3:H\to \mathrm{Diff}(\mathsf{E}),$ defined by $\chi_3(g) = \lambda_3(h),$ is a monomorphism.

    Let $\sigma':H\times \mathsf{E}\to \mathsf{E}$ be the map given by the formula
    $$
    \sigma'(g, x) = \chi_3(g)(x),\quad x\in\mathsf{E}.
    $$
    Since $\chi_3$ is a homomorphism, it follows that $\sigma'$ is an $H$-action on $\mathsf{E}.$

    Hence, we define an $H$-action $\Sigma:H\times M\to M$ on $M$ by the rule:
    $$
    \Sigma = \begin{cases}
    \sigma',& \text{on }  H\times\mathsf{E},\\
    \sigma, & \text{on } H\times\mathcal{E}_{V,a}.
    \end{cases}
    $$
    Theorem \ref{th:main} is proved.
\end{proof}

\subsection*{Proof of Lemma \ref{lm:ch2}}
Due to \cite[Lemma 4.12.]{Maktymenko:orb-stab:200}, for the diffeomorphism $h'_{ri} =$ \break  $\chi_2^{-1}(g)|_{Y_{ri}}\circ h_{ri}|_{Y_r}:Y_r\to Y_r$ of the cylinder there exists a smooth function $\alpha_{ri}$ such that $h'_{ri} = \mathbf{F}_{\alpha_{ri}}$ whenever
for each trajectory $\omega$ of $\mathbf{F}$ we have that $h'_{ri}(\omega) = \omega$ and $h'_{ri}$ preserves orientation of $\omega$.

Let $\omega$ be a trajectory of $\mathbf{F}.$
It follows from condition (C) that
$\omega = f^{-1}(t)\cap Y_r$ for some $t\in\RRR.$ The sets $f^{-1}(t)$ and $Y_r$ are $h'_{ri}$-invariant, so the set $f^{-1}(t)\cap Y_r$ is also $h'_{ri}$-invariant. Hence,  $h_{ri}'(\omega) = \omega$ for all trajectories of $\mathbf{F}.$ Moreover, by Lemma \ref{lm:part}, $h'_{ri}$ preserves orientations of each orbit $\omega$ of $\mathbf{F}.$ Thus $h'_{ri} = \mathbf{F}_{\alpha_{ri}}$, and due to (\ref{eq;invers}) we put $\xi_{ri} = -\alpha_{ri}\circ \mathbf{F}_{\alpha_{ri}}^{-1}.$
Lemma is proved.
\subsection*{Proof of  Lemma \ref{lm:ext}}
By the result of Seeley \cite{Seeley:1964}, the function $\xi_{ri}$ extends to  some smooth function $\beta'_{ri}$ on $E_{ri}.$
It is easy to construct  a function $\delta_{r}\in C^{\infty}(E_{r0},[0,1])$ which satisfies the following conditions:
\begin{enumerate}
    \item $\delta_r = 1$ on $Y_r$,
    \item $\delta_r = 0$ on some neighborhood of $Z_r\cap (f^{-1}(-a_1)\cup f^{-1}(a_1))$.
    \item $F\delta_r  = 0$, i.e., $\delta_r$ is constant along orbits of $F,$
    \item the function $\beta_{ri}  = \delta_r\beta'_{ri}$, $i = 1,2,\ldots, n$ satisfies the inequality $F\beta_{ri}|_{Z_{r}} > -1$.
\end{enumerate}
Indeed, since $\beta'_{ri} = \xi_{ri}$ on $Y_{ri}$ and $\mathbf{F}_{\alpha_{ri}}$ is a diffeomorphism, it follows that $F\beta'_{ri}> -1$ on $Y_{ri}.$ Then there exists $b\in (a, a_1)$ such that  $F\beta'_{ri} > -1$ on   $A_r = E_{r0}\cap f^{-1}([-b, -a/2]\cup [a/2,b])\cap \mathcal{E}_{V,a}$.  Let $\delta_r:E_{r0}\to [0,1]$ be a smooth function such that $\delta_r = 1$ on $Y_r$, $\delta_r = 0$ on $E_{r0}\setminus A_{r},$ and $F\delta_r = 0.$ Then $\delta F\beta'_{ri} > -1$ on $E_{ri}.$
Now the required diffeomorphism $w_{ri}:{E_{r0}}\to E_{r0}$ can be defined by the formula
$$
w_{ri}(x) = \begin{cases}
\mathbf{F}_{\beta_{ri}}, & x\in Z_{r},\\
x, & x\in E_{r0}\setminus Z_r.
\end{cases}
$$
Lemma is proved.

{\it {Acknowledgments.}}
The author is grateful to Sergiy Maksymenko and Eugene Polu\-lyakh for useful discussions.

\end{document}